\documentclass[12pt]{amsart}
\usepackage[margin=1in]{geometry}

\usepackage{graphicx}
\usepackage{amsthm}
\usepackage{amssymb}
\usepackage[table]{xcolor}

\usepackage{caption}
\usepackage{enumitem}
\usepackage[hidelinks]{hyperref}

\newtheorem{theorem}{Theorem}
\newtheorem{prop}[theorem]{Proposition}
\newtheorem{lemma}[theorem]{Lemma}

\newtheorem{conj}[theorem]{Conjecture}
\theoremstyle{definition}
\newtheorem{definition}[theorem]{Definition}

\theoremstyle{remark}

\newtheoremstyle{named}{}{}{\itshape}{}{\bfseries}{.}{.5em}{\thmnote{#3}}
\theoremstyle{named}
\newtheorem*{namedtheorem}{}

\numberwithin{theorem}{section}
\numberwithin{equation}{section}

\newcommand{\R}{\mathbb{R}}

\newcommand{\mb}{\mathbb}

\renewcommand{\phi}{\varphi}

\newcommand{\paren}[1]{\left(#1\right)}
\newcommand{\set}[1]{\left\{#1\right\}}

\newcommand{\abs}[1]{\left\lvert#1\right\rvert}

\newcommand{\Vtil}{\widetilde{V}}

\begin{document}

\title{1D Triple Bubble Problem with Log-Convex Density}

\author{Nat Sothanaphan}
\email{natsothanaphan@gmail.com}

\begin{abstract}
We prove that for a symmetric, strictly log-convex density on the real line, there are four possible types of perimeter-minimizing triple bubbles. This extends the work of Bongiovanni et al. \cite{Bo}, which shows that there are
two possible types of perimeter-minimizing double bubbles.
\end{abstract}

\maketitle


\section{Introduction}

The log-convex density theorem proved by Chambers \cite{Ch} states that on $\R^N$ with smooth, radially symmetric, log-convex density, a sphere centered at the origin is perimeter minimizing for a given volume.
We seek the corresponding least-perimeter way to enclose and separate many given volumes, the perimeter-minimizing $n$-bubble.
In our previous work \cite{Bo}, we showed that for double bubbles on $\R$ with symmetric, strictly log-convex density, there are two possible types of perimeter minimizers shown in Figure \ref{fig:doubletriple}.

\begin{figure}[h]
\includegraphics[width=180pt]{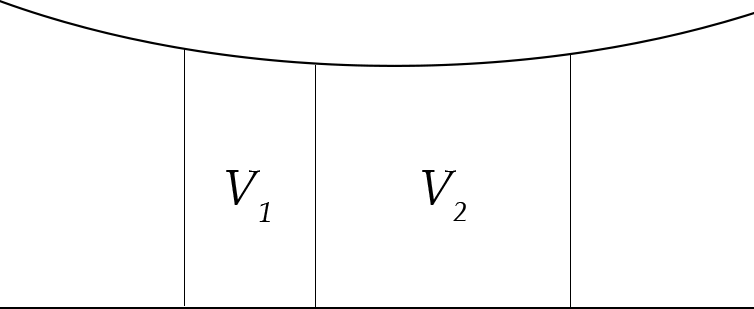}
\hspace{20pt}
\includegraphics[width=180pt]{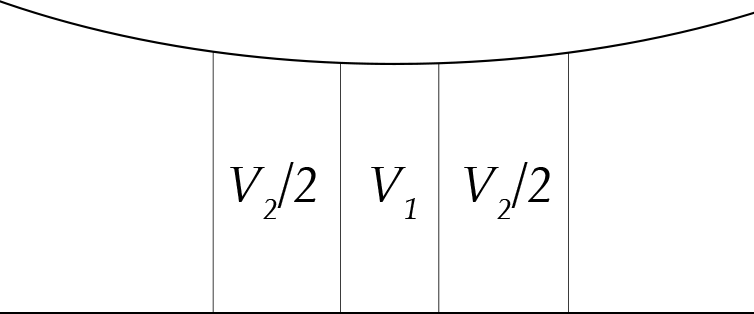}
\caption{The two types of perimeter-minimizing double bubbles on $\R$ with symmetric, strictly log-convex density for given volumes $V_1 \leq V_2$.}
\label{fig:doubletriple}
\end{figure}

This paper extends the techniques of our previous work \cite{Bo} to characterize triple bubbles
on $\R$ with symmetric, strictly log-convex density.
Our main result is stated in Theorem \ref{thm:main}: there are four possible types of perimeter-minimizing triple bubbles illustrated in Figure \ref{fig:4types}.
We also prove that minimizers of each of the four types exist using numerical computation.

\begin{figure}[h]
\includegraphics[width=350pt]{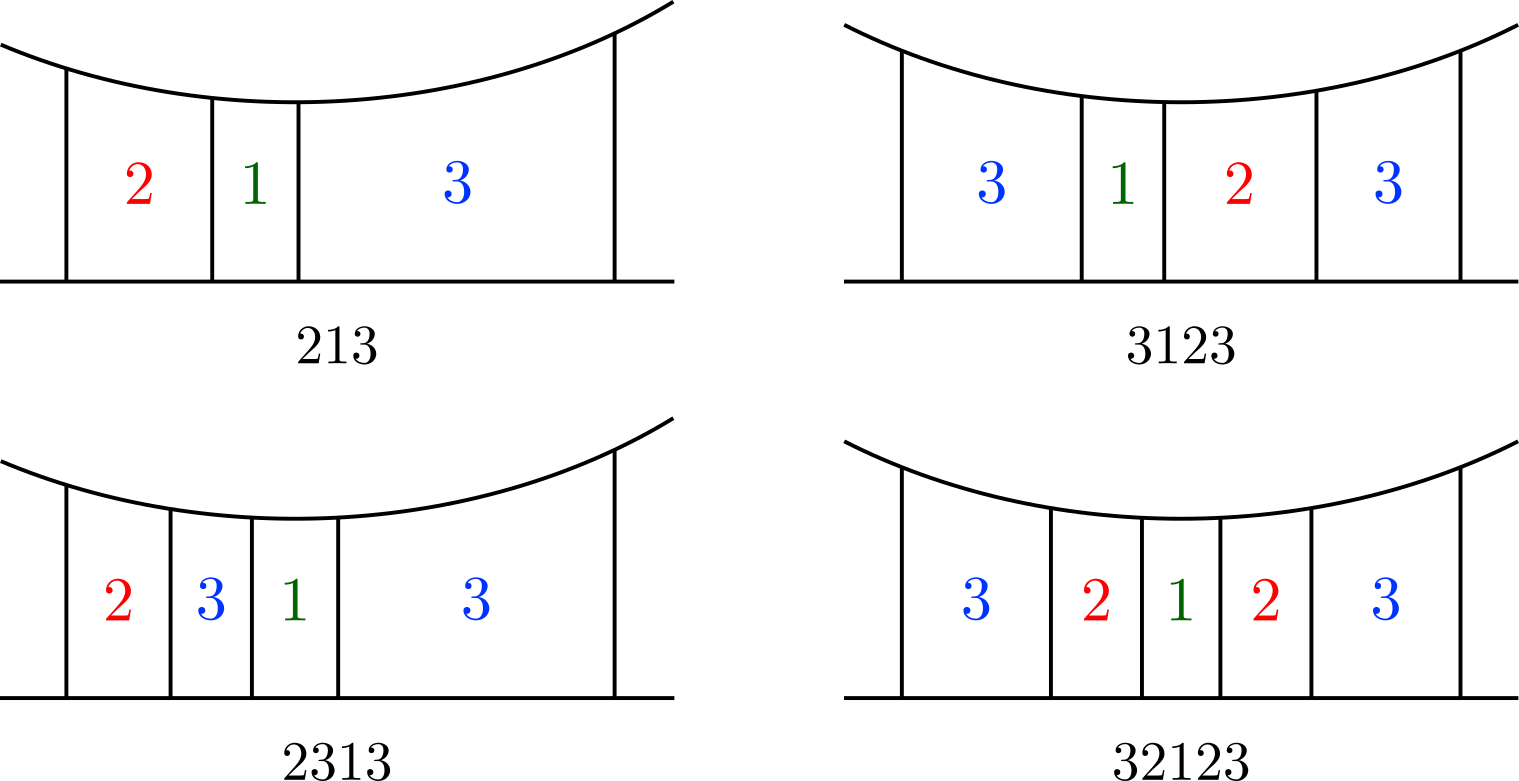}
\caption{The four types---213, 3123, 2313, and 32123---of perimeter-minimizing triple bubbles with prescribed volumes
$V_1 \leq V_2 \leq V_3$. The label $n$ means that
the corresponding component is part of the bubble with volume $V_n$.}
\label{fig:4types}
\end{figure}

\begin{theorem}
\label{thm:main}
On $\R$ with symmetric, log-convex density that is uniquely minimized at the origin and prescribed volumes $V_1 \leq V_2 \leq V_3$, a perimeter-minimizing triple bubble is, up to reflection, of one of the four types: $213$, $3123$, $2313$, and $32123$, as shown in Figure \ref{fig:4types}. Moreover, there exists a symmetric, strictly log-convex, $C^1$ density such that each of the four types is perimeter minimizing for some prescribed volumes.
\end{theorem}

\begin{figure}[h]
\includegraphics[width=225pt]{{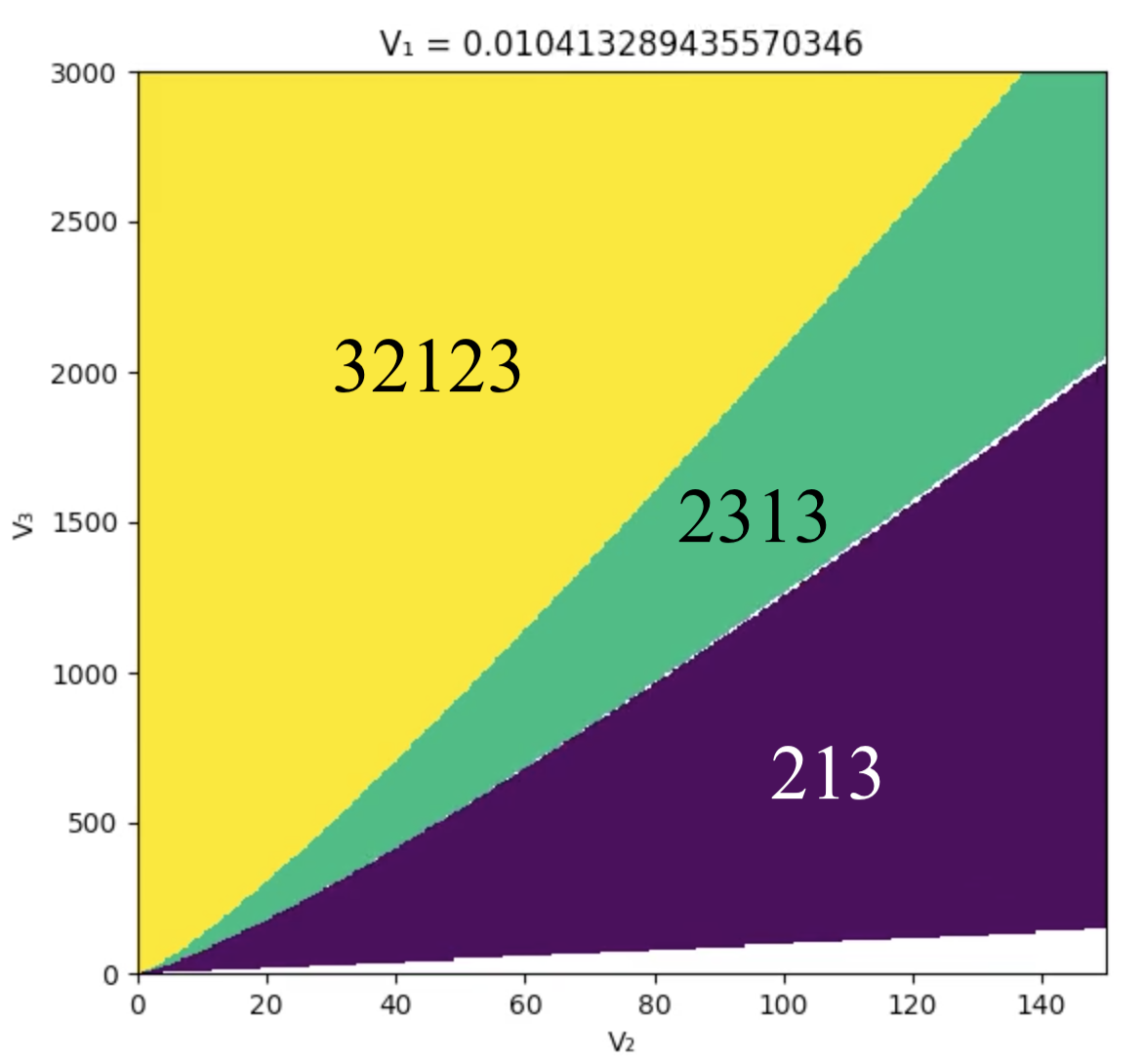}}
\includegraphics[width=225pt]{{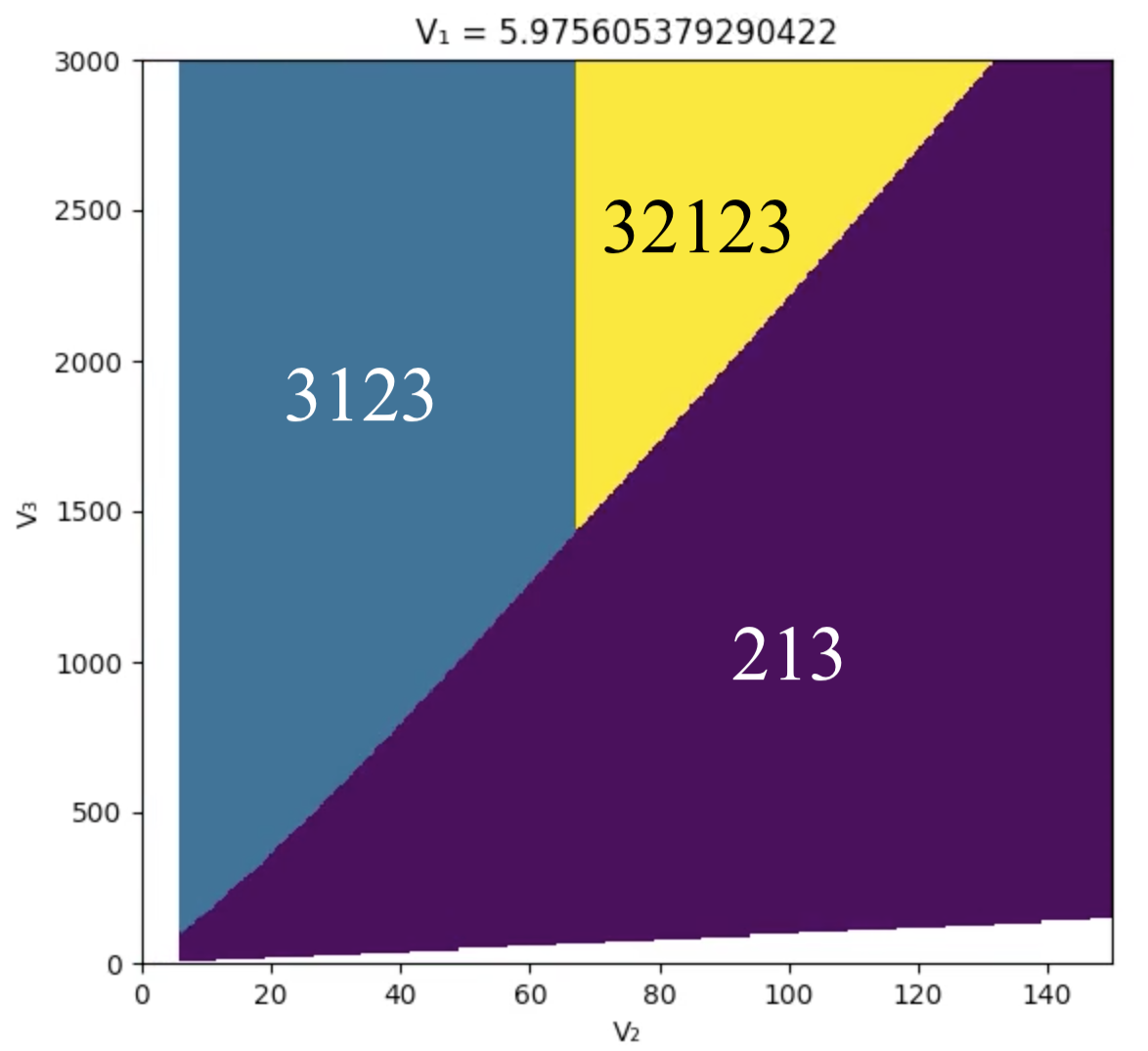}}
\caption{Numerical proof (conclusive) of the existence of four types of minimizers for the density $f_2$ of Section \ref{sec:code}.
Colors correspond to types as: \textcolor{violet}{purple} = 213; \textcolor{blue}{blue} = 3123; \textcolor{teal}{green} = 2313; \textcolor{orange}{yellow} = 32123.
See more information in Figure \ref{fig:seconddensity}.}
\label{fig:previewplot}
\end{figure}

The isoperimetric problem with density was studied by Bobkov and Houdr\'{e} \cite{BH}, who studied its connections with Sobolev-type inequalities, and Bayle \cite{Ba}, who studied isoperimetric profiles.
An important case is $\R^N$ with smooth, radially symmetric, log-convex density, where the log-convex density theorem, conjectured by Rosales et al. \cite{RCBM} and proved by Chambers \cite{Ch}, states that spheres centered at the origin are isoperimetric.

The double bubble theorem \cite[Chapt. 14]{Mo} states that in $\R^N$, the standard double bubble
consisting of three spherical caps meeting at 120 degrees is the least-perimeter double bubble.
This was proven in $\R^3$ by Hutchings et al. \cite{HMRR} and in $\R^N$ by Reichardt \cite{Re}.
There are results on double bubbles in the sphere $\mb{S}^N$, hyperbolic space $\mb{H}^N$, flat tori $\mb{T}^2$ and $\mb{T}^3$, and Gauss space (Euclidean space with density $e^{-r^2}$); see \cite[Chapt. 19]{Mo}. Recently, Milman and Neeman \cite{MN} proved the Gaussian double bubble conjecture, which states that the solution is three half-hyperplanes meeting at 120 degrees.

For more than two bubbles, Wichiramala \cite{Wi} proved the triple bubble conjecture in $\R^2$ that the standard triple bubble is perimeter minimizing. Milman and Neeman \cite{MN2} also recently proved the Gaussian multi-bubble conjecture.

Bongiovanni et al. \cite{Bo} solved the double bubble problem on $\R$ with smooth, symmetric, strictly log-convex density $f$ and characterized the transitions of minimizers when $(\log f)'$ is unbounded.
The author \cite{So} extended the characterization to the case where $(\log f)'$ is bounded.

This paper is organized as follows. Section \ref{sec:nstruct} gives results on structures of $n$-bubbles
on $\R$ that improve on those of Bongiovanni et al. \cite{Bo}.
In particular, we show that a perimeter-minimizing $n$-bubble must be a ``standard nested bubble,'' defined in Definition \ref{def:nested}.

In Section \ref{sec:types}, under the mild condition that the density is nonincreasing on $(-\infty,0]$ and nondecreasing on $[0,\infty)$, we narrow down the types of perimeter-minimizing triple bubbles to ten possible types.
Some arguments in this section are based on personal communication with Antonio Ca\~{n}ete.

Section \ref{sec:types2} proves our main result that there are four types of perimeter-minimizing triple bubbles
under the hypotheses of Theorem \ref{thm:main}. We show this result by arguments based on rearrangements and equilibrium conditions and draw from tools developed by Bongiovanni et al. \cite{Bo}.

Section \ref{sec:code} gives numerical proof of the existence of all four types of minimizers
for a certain symmetric, strictly log-convex density. The code used for this purpose is more efficient than the one used in our previous work \cite{Bo}. The plots obtained will also be useful in forming
conjectures in Section \ref{sec:conj}.

Finally, we state in Section \ref{sec:conj} conjectures on transitions of minimizers as volumes vary,
which may serve as a basis for further work.

Results from \cite{Bo} that are used in this work are collected in the Appendix.

\subsection*{Acknowledgments}
I would like to thank Antonio Ca\~{n}ete, Eliot Bongiovanni, and Frank Morgan for discussions on this problem.
Antonio Ca\~{n}ete initiated the discussion, contributed questions and arguments, and gave comments on a draft of this paper; Eliot Bongiovanni came up with conjectures and made progress in various directions; and Frank Morgan gave comments on my arguments and the writing of this paper.

\section{$n$-Bubble Structure}
\label{sec:nstruct}

In this section, we study perimeter-minimizing $n$-bubbles on $\R$ with a density that is nonincreasing on $(-\infty,0]$ and nondecreasing on $[0,\infty)$.
Our results show that a perimeter-minimizing $n$-bubble must be a ``standard nested bubble'': standard multi-bubbles contained within one another and defined in Definition \ref{def:nested}.
This improves on the results of Bongiovanni et al. \cite[Prop. 3.8]{Bo}.
Figure \ref{fig:nested2d} shows an example of a 2D nested bubble, and Figure \ref{fig:nested1d} gives an example of a 1D nested bubble.

\begin{figure}[h]
\includegraphics[width=200pt]{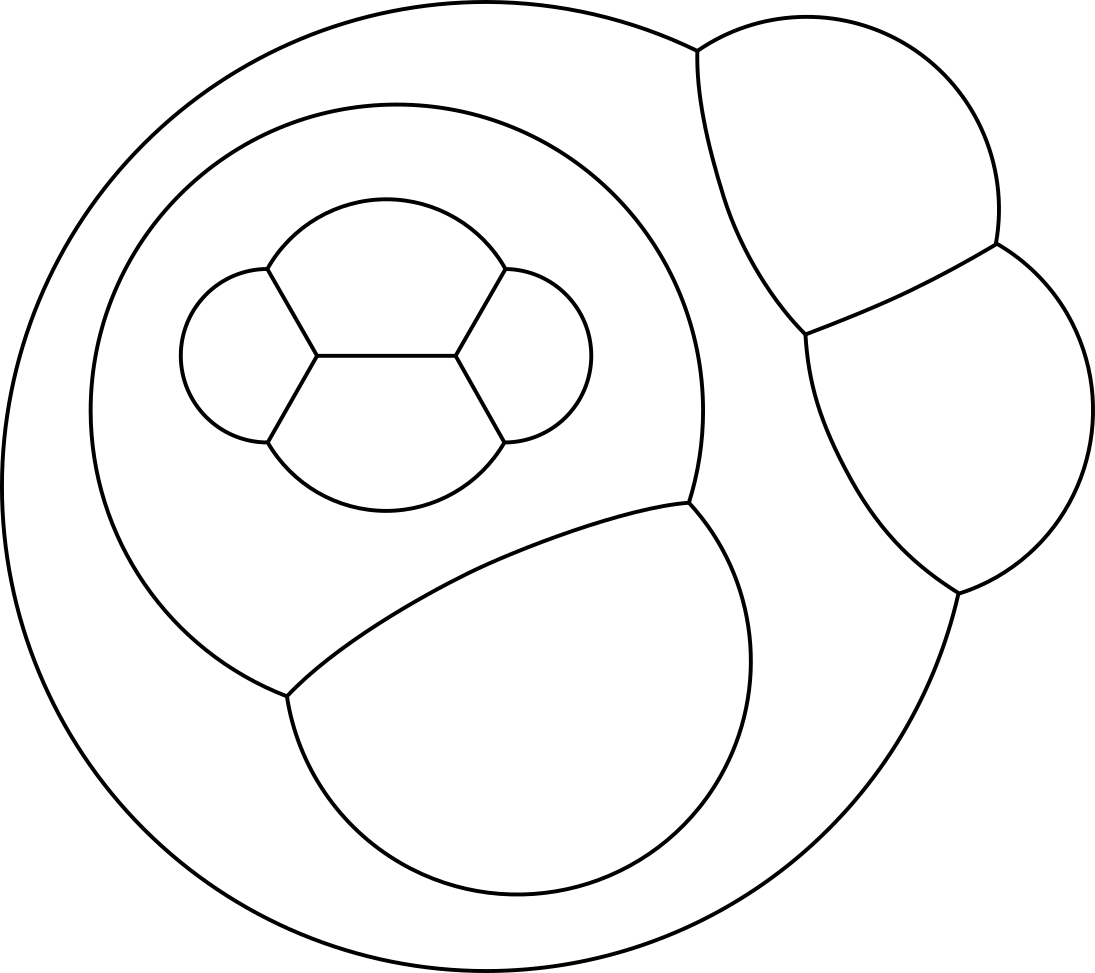}
\caption{2D nested bubble of type (3,2,4).}
\label{fig:nested2d}
\end{figure}

\begin{figure}[h]
\includegraphics[width=330pt]{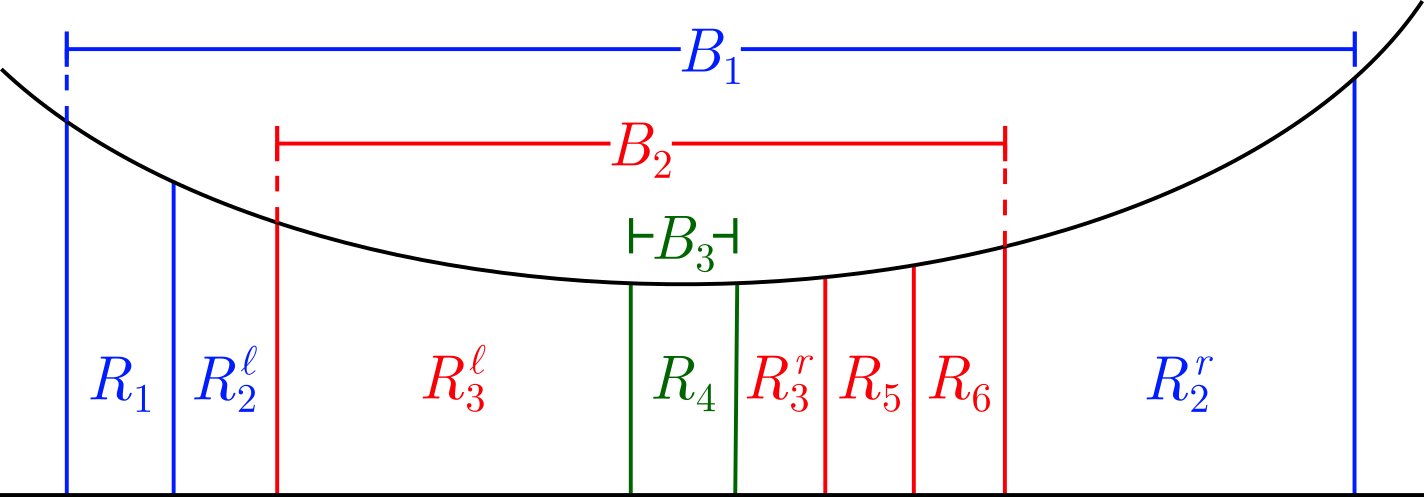}
\caption{1D standard nested bubble of type $(2,3,1)$, defined in Definition \ref{def:nested}.}
\label{fig:nested1d}
\end{figure}

We start with definitions. Consider a perimeter-minimizing $n$-bubble on $\R$ with prescribed volumes $V_1 \leq V_2 \leq \cdots \leq V_n$. The bubble with volume $V_a$ is called ``bubble $a$.'' If bubble $a$ has one component, let $R_a$ denote this component. If bubble $a$ has two components, let $R_a^\ell$ and $R_a^r$ denote the left and right components and $V_a^\ell$ and $V_a^r$ denote their volumes, respectively.
The following two definitions are only for 1D $n$-bubbles.

\begin{definition}
On $\R$ with density, a \emph{standard $n$-bubble} is an $n$-bubble with contiguous components in equilibrium where every bubble has one component.
Here, ``in equilibrium'' means being a local minimizer of the perimeter functional.
\end{definition}

\begin{definition}
\label{def:nested}
Let $k_1,\dots,k_m$ be positive integers.
An $n$-bubble $B$ is a \emph{standard nested bubble of type $(k_1,\dots,k_m)$}
if there are standard $k_i$-bubbles $B_i$ with disjoint boundaries such that
\begin{enumerate}
\item The boundary points of $B$ consist exactly of the boundary points of the $B_i$'s;
\item $B_{i+1}$ is contained in a single component of $B_i$;
\item In $B$, the components to the left and right of $B_i$ for $i>1$ are a single bubble, and no other bubbles of $B$ have more than one component.
\end{enumerate}
\end{definition}

See Figure \ref{fig:nested1d}. Notice that for a standard nested $n$-bubble of type $(k_1,\dots,k_m)$, the $k_i$'s are uniquely determined and sum to $n$. The triple interval \cite[Def. 4.1]{Bo} on the right-hand side of Figure \ref{fig:doubletriple} is a standard nested bubble of type $(1,1)$.

It is easily seen that the requirements for an $n$-bubble $B$ to be a standard nested bubble is equivalent to the following.
\begin{enumerate}[label=(\roman*)]
\item $B$ consists of contiguous intervals;
\item Every bubble of $B$ has at most two components;
\item If bubbles $a$ and $b$ both have two components, then the order of their components is either $R_a^\ell,R_b^\ell,R_b^r,R_a^r$
or $R_b^\ell,R_a^\ell,R_a^r,R_b^r$.
\end{enumerate}

We now prove the following proposition on the structure of 1D perimeter-minimizing $n$-bubbles.

\begin{prop}
\label{prop:1dstruct}
On $\R$ with continuous density that is nonincreasing on $(-\infty,0]$ and nondecreasing on $[0,\infty)$, a perimeter-minimizing $n$-bubble is a standard nested bubble.
\end{prop}

\begin{proof}
We show that any perimeter-minimizing $n$-bubble satisfies conditions (i)--(iii) stated above.
Point (i) follows from \cite[Prop. 3.5]{Bo}, and point (ii) follows from  \cite[Lemma 3.7]{Bo}.
It remains to verify point (iii).

Without loss of generality, let $R_a^\ell$ be the leftmost of the four components. By \cite[Lemma 3.7]{Bo}, any left component must be to the left of any right component, so the order $R_a^\ell,R_a^r,R_b^\ell,R_b^r$ is impossible.

Finally, we show that an $n$-bubble with the order $R_a^\ell,R_b^\ell,R_a^r,R_b^r$ is not perimeter-minimizing by an argument similar to the proof of \cite[Prop. 3.8]{Bo}.
Move all components and boundary points between $R_a^\ell$ and $R_b^\ell$ to the right and all components and boundary points between $R_a^r$ and $R_b^r$ to the left in such a way as to preserve volumes of bubbles.
Observe that moving to the right increases the volume of bubble $a$ and decreases the volume of bubble $b$, and vice versa for moving to the left. Hence matching the velocities of moving to the right and to the left allows us to preserve volumes of bubbles $a$ and $b$.

By \cite[Lemma 3.7]{Bo}, perimeter does not increase during the move. But when one of $R_b^\ell$ or $R_a^r$ disappears, the perimeter will decrease. Therefore the only possible order is $R_a^\ell,R_b^\ell,R_b^r,R_a^r$.
\end{proof}

Proposition \ref{prop:1dstruct} indeed improves on our previous characterizations of perimeter-minimizing $n$-bubbles,
as it immediately implies \cite[Prop. 3.8]{Bo}.

We close this section with two conjectures. First, we think that the characterization in Proposition \ref{prop:1dstruct} is optimal in the sense that a standard nested bubble of every type can be perimeter minimizing.
This is stated in Conjecture \ref{conj:structoptimal}.
Our previous work \cite[Thm. 4.15]{Bo} proved this for double bubbles and Section \ref{sec:code} proves it for triple bubbles.

\begin{conj}
\label{conj:structoptimal}
For any positive integers $n$ and $k_1,\dots,k_m$ that sum to $n$, a standard nested bubble of type $(k_1,\dots,k_m)$ is a perimeter-minimizing $n$-bubble for some symmetric, log-convex density on $\R$ and some prescribed volumes.
\end{conj}

Finally, Proposition \ref{prop:1dstruct} might be true in higher dimensions for smooth, radially symmetric, strictly log-convex densities, for an appropriate notion of ``standard nested bubble.''
This is stated in Conjecture \ref{conj:structhighdim}.
The case of single bubbles is the log-convex density theorem \cite{Ch}, and the case of double bubbles was conjectured in our previous work \cite[Conj. 7.1]{Bo}.

\begin{conj}
\label{conj:structhighdim}
On $\R^N$ with smooth, radially symmetric, strictly log-convex density, a perimeter-minimizing $n$-bubble is a standard nested bubble.
\end{conj}

\section{Triple Bubbles Under Mild Conditions on Density}
\label{sec:types}

We now focus on perimeter-minimizing triple bubbles on $\R$.
Let the notation
$$\mid I_1 \mid I_2 \mid \dots \mid I_m \mid,$$
where $I_i$'s are compact intervals, denote a multi-bubble with contiguous components $I_1,I_2,\dots,I_m$ from left to right in this order.
By Proposition \ref{prop:1dstruct} applied to triple bubbles, the following is immediate.

\begin{prop}
\label{prop:init}
On $\R$ with continuous density that is nonincreasing on $(-\infty,0]$ and nondecreasing on $[0,\infty)$, a perimeter-minimizing triple bubble is, up to reflection, of one of the forms:
\begin{align*}
&(1) \mid R_a \mid R_b \mid R_c \mid; &&(2) \mid R_a^\ell \mid R_b \mid R_c \mid R_a^r \mid; \\
&(3)\mid R_a \mid R_b^\ell \mid R_c \mid R_b^r \mid; &&(4) \mid R_a^\ell \mid R_b^\ell \mid R_c \mid R_b^r \mid R_a^r \mid,
\end{align*}
where $\set{a,b,c}=\set{1,2,3}$.
\end{prop}

To further reduce the number of possibilities, we will use the following lemma.
The idea for this lemma comes from arguments of Antonio Ca\~{n}ete.

\begin{lemma}[Interval Squeezing]
\label{lem:squeeze}
Let $f$ be a positive function that is nonincreasing on $(-\infty,0]$ and nondecreasing on $[0,\infty)$.
Suppose that $x_0 \leq x_1 \leq \dots \leq x_n$ with $y_i \in [x_{i-1},x_i]$.
Then $f(x_0)+\dots+f(x_n) > f(y_1)+\dots+f(y_n)$.
\end{lemma}

\begin{proof}
We proceed by induction. The case $n=0$ is obvious.
For $n \geq 1$, if $y_1 \geq 0$, then $f(x_i) \geq f(y_i)$ and $f(x_0)>0$, so the inequality holds.
If $y_1 < 0$, then $f(x_0)\geq f(y_1)$, and we can apply the induction hypothesis.
\end{proof}

By using Lemma \ref{lem:squeeze}, we can strengthen Proposition \ref{prop:init} as follows.

\begin{prop}
\label{prop:typemonotone}
On $\R$ with continuous density that is nonincreasing on $(-\infty,0]$ and nondecreasing on $[0,\infty)$, a perimeter-minimizing triple bubble is, up to reflection, of one of the forms:
\begin{enumerate}[font=\upshape]
\item  $\mid R_2 \mid R_1 \mid R_3 \mid$, $\mid R_1 \mid R_2 \mid R_3 \mid$,
$\mid R_1 \mid R_3 \mid R_2 \mid$;
\item $\mid R_3^\ell \mid R_1 \mid R_2 \mid R_3^r \mid$, $\mid R_2^\ell \mid R_1 \mid R_3 \mid R_2^r \mid$;
\item $\mid R_2 \mid R_3^\ell \mid R_1 \mid R_3^r \mid$, $\mid R_1 \mid R_3^\ell \mid R_2 \mid R_3^r \mid$, $\mid R_3 \mid R_2^\ell \mid R_1 \mid R_2^r \mid$;
\item $\mid R_3^\ell \mid R_2^\ell \mid R_1 \mid R_2^r \mid R_3^r \mid$,
$\mid R_2^\ell \mid R_3^\ell \mid R_1 \mid R_3^r \mid R_2^r \mid$.
\end{enumerate}
\end{prop}

\begin{proof}
We need to rule out some cases from Proposition \ref{prop:init} as not perimeter minimizing.

(2): We eliminate $\mid R_1^\ell \mid R_2 \mid R_3 \mid R_1^r \mid$. Let the boundary points be $x_0<x_1<x_2<x_3<x_4$.
Switch $R_1^\ell$ with $R_2$ and $R_3$ with $R_1^r$ to get $\mid R_2 \mid R_1 \mid R_3 \mid$ with boundary points $x_0<y_1<y_2<x_4$.
Then $y_1 \in [x_1,x_2]$ and $y_2 \in [x_2,x_3]$. By Lemma \ref{lem:squeeze}, $f(x_1)+f(x_2)+f(x_3) > f(y_1)+f(y_2)$, so the original configuration is not perimeter minimizing.

(3): We rule out $\mid R_a \mid R_b^\ell \mid R_c \mid R_b^r \mid$ with $b < c$.
Switch $R_c$ and $R_b^r$ to get $\mid R_a \mid R_b \mid R_c \mid$.
Then two boundary points disappear and one new boundary point appears between them.
By Lemma \ref{lem:squeeze}, the new configuration has less perimeter than the original one.

(4): We rule out $\mid R_a^\ell \mid R_b^\ell \mid R_c \mid R_b^r \mid R_a^r \mid$ with $b<c$, and $\mid R_1^\ell \mid R_3^\ell \mid R_2 \mid R_3^r \mid R_1^r \mid$.
For the first one, switch $R_c$ and $R_b^r$ and use a similar argument as in case (3).
For the second one, assume by symmetry that $V_1^\ell \leq V_3^r$.
Then compare with $\mid R_2 \mid R_3 \mid R_1 \mid$ with the same leftmost and rightmost boundary points
using Lemma \ref{lem:squeeze}.
To make the comparison, ignore the right endpoint of $R_3^\ell$.
\end{proof}

We do not know whether more types can be eliminated from Proposition \ref{prop:typemonotone} without more conditions on the density.
For completeness, we include a double bubble counterpart of Proposition \ref{prop:typemonotone} that generalizes \cite[Prop. 4.6]{Bo}.

\begin{prop}
\label{prop:doublemonotone}
On $\R$ with continuous density that is nonincreasing on $(-\infty,0]$ and nondecreasing on $[0,\infty)$, a perimeter-minimizing double bubble is, up to reflection, of the form 
$$\mid R_1 \mid R_2 \mid \quad \text{or} \quad \mid R_2^\ell \mid R_1 \mid R_2^r \mid.$$
\end{prop}

\begin{proof}
By Proposition \ref{prop:1dstruct}, a perimeter minimizer must be a standard nested bubble.
Thus we only need to rule out $\mid R_1^\ell \mid R_2 \mid R_1^r \mid$.
Indeed, switching $R_2$ and $R_1^r$ yields a configuration with less perimeter by Lemma \ref{lem:squeeze}.
\end{proof}

Notice that in Propositions \ref{prop:typemonotone} and \ref{prop:doublemonotone}, the smallest bubble always has one component.
We suspect that this may be true for $n$-bubbles for any $n$, as stated in the following conjecture.

\begin{conj}
On $\R$ with continuous density that is nonincreasing on $(-\infty,0]$ and nondecreasing on $[0,\infty)$, in a perimeter-minimizing $n$-bubble, the smallest volume, if it is unique, is contained in one component.
\end{conj}

\section{Triple Bubbles Under Log-Convex Density}
\label{sec:types2}

Our goal in this section is to reduce the ten types of Proposition \ref{prop:typemonotone} to the four types of Figure \ref{fig:4types}.
To do so, we assume that the density is symmetric, log-convex, and uniquely minimized at the origin.
Symmetry and log-convexity are not enough: the type $\mid R_1 \mid R_2 \mid R_3 \mid$ is perimeter minimizing for a constant density.
It is easily seen that this condition is weaker than symmetry and strict log-convexity.
It is in fact strictly weaker, as the density $e^{\abs{x}}$ satisfies it but is not strictly log-convex.
An equivalent condition is that the density is symmetric, log-convex, strictly decreasing on $(-\infty,0]$, and strictly increasing on $[0,\infty)$.

Recall that a convex function $g$ has left and right derivatives, which we denote by $g'_L$ and $g'_R$.
Given a density $f$, define the volume coordinate by
$$V= \int_0^x f,$$
where $x$ is the positional coordinate.
Since $x \mapsto V$ is strictly monotone increasing, this is a valid change of coordinates, and we can view $f$ as a function of either $x$ or $V$.
By \cite[Lemma 4.2]{Bo}, $f$ is log-convex if and only if it is \emph{convex} in volume coordinates,
in which case $(\log f)'_L(x) = f'_L(V)$ and $(\log f)'_R(x) = f'_R(V)$.

The following lemma is a piece of logic that will be used in Proposition \ref{prop:typelogconvex}.

\begin{lemma}
\label{lem:summorethanzero}
Let $f$ be a symmetric and log-convex density on $\R$.
In volume coordinates, $f'_R(v_1)+f'_R(v_2)>0$ implies $v_1+v_2 \geq 0$,
and $v_1+v_2 > 0$ implies $f'_L(v_1)+f'_L(v_2) \geq 0$.
\end{lemma}

\begin{proof}
Observe that $f'_L(-v)=-f'_R(v)$ by symmetry and $v>w$ implies $f'_L(v) \geq f'_R(w)$ by convexity in volume coordinates. Hence
\begin{align*}
v_1+v_2<0 &\implies -v_2>v_1 \implies f'_L(-v_2) \geq f'_R(v_1)  \\
&\implies -f'_R(v_2) \geq f'_R(v_1) \implies f'_R(v_1)+f'_R(v_2) \leq 0,
\end{align*}
implying the first conclusion. The second conclusion follows by replacing $v_1$ and $v_2$ with $-v_1$ and $-v_2$.
\end{proof}

We now show that for a symmetric, log-convex density that is uniquely minimized at the origin, only one type from each group in Proposition \ref{prop:typemonotone} can be perimeter minimizing.

\begin{prop}
\label{prop:typelogconvex}
On $\R$ with symmetric, log-convex density that is uniquely minimized at the origin, a perimeter-minimizing triple bubble is, up to reflection, of one of the forms:
\begin{align*}
&(1) \mid R_2 \mid R_1 \mid R_3 \mid; && (2) \mid R_3^\ell \mid R_1 \mid R_2 \mid R_3^r \mid; \\
&(3) \mid R_2 \mid R_3^\ell \mid R_1 \mid R_3^r \mid; && (4) \mid R_3^\ell \mid R_2^\ell \mid R_1 \mid R_2^r \mid R_3^r \mid.
\end{align*}
\end{prop}

\begin{proof}
We need to eliminate more types from Proposition \ref{prop:typemonotone} as not perimeter minimizing.
Let $f$ be the density. Recall that $f$ is strictly decreasing on $(-\infty,0]$ and strictly increasing on $[0,\infty)$.
Notice that, in volume coordinates, $f'_L(v)>0$ if $v>0$ and $f'_R(v)<0$ if $v<0$.

(4): We rule out $\mid R_2^\ell \mid R_3^\ell \mid R_1 \mid R_3^r \mid R_2^r \mid$ with $V_2<V_3$.
By \cite[Lemma 3.7]{Bo}, $0 \in R_1$.
We now try switching bubbles 2 and 3.
Consider the configuration $\mid R_3^{\ell\prime} \mid R_2^{\ell\prime} \mid R_1 \mid R_2^{r\prime} \mid R_3^{r\prime} \mid$, where the outer boundary points and $R_1$ do not move, and the volumes are split so that $V_2^\ell < 
V_3^{\ell\prime} < V_2^\ell+V_3^\ell$ and $V_2^r < V_3^{r\prime} < V_2^r +V_3^r$.
By monotonicity of $f$ on $(-\infty,0]$ and $[0,\infty)$, this configuration has less perimeter than the original one.

(3): We show that if $\mid R_a \mid R_b^\ell \mid R_c \mid R_b^r \mid$ is perimeter minimizing, then $V_b \geq V_a \geq V_c$.
By \cite[Lemma 3.7]{Bo}, $0 \in R_c$.
Switching $R_a$ and $R_b^\ell$ shows that $V_a \geq V_b^\ell$.
Switching $R_b^\ell$ with $R_c$ allows us to conclude by Lemma \ref{lem:squeeze} that $V_b^\ell \geq V_c$.
Thus $V_a \geq V_b^\ell \geq V_c$.

It remains to show that $V_b \geq V_a$.
In volume coordinates, let the boundary points be $w_0,w_1,w_2,w_3,w_4$.
By the equilibrium condition \cite[Rmk. 3.4]{Bo},
$$f'_R(w_0)+f'_R(w_1)+f'_R(w_4) \geq 0.$$
Since $0 \in R_c$, $w_1 < 0$, so $f'_R(w_1)<0$ and hence
$f'_R(w_0)+f'_R(w_4)>0$. By Lemma \ref{lem:summorethanzero},
$w_0+w_4 \geq 0$.
Since $w_2 \leq 0$, $w_0 \leq -V_a-V_b^\ell$ and $w_4 \leq V_c+V_b^r$, and so $V_c+V_b^r \geq V_a+V_b^\ell$.
Finally, $V_c \leq V_b^\ell$ implies that $V_b^r\geq V_a$, so $V_b \geq V_a$ as required.

(2): We show that if $\mid R_a^\ell \mid R_b \mid R_c \mid R_a^r \mid$ is perimeter minimizing, then  $V_a \geq V_b$
and $V_a \geq V_c$.
By \cite[Lemma 3.7]{Bo}, $0 \in R_b \cup R_c$. Assume by symmetry that $0 \in R_b$.
Switching $R_c$ and $R_a^r$ gives $V_a^r \geq V_c$, and so $V_a \geq V_c$.

It remains to show that $V_a \geq V_b$.
Let the endpoints in volume coordinates be $w_0,w_1,w_2,w_3,w_4$.
By switching $R_a^\ell$ and $R_b$ and noting that the new boundary point must be no closer to the origin than the old one, either $V_a^\ell \geq V_b$ or $w_2-V_a^\ell \geq -w_1$. In the first case, $V_a \geq V_b$, as desired. In the second case, $w_1+w_2\geq V_a^\ell >0$.
By Lemma \ref{lem:summorethanzero}, $f'_L(w_1)+f'_L(w_2) \geq 0$.
Because $w_3>0$, $f'_L(w_3)>0$.
Combining these gives $f'_L(w_1)+f'_L(w_2)+f'_L(w_3)>0$, contradicting the equilibrium condition $f'_L(w_1)+f'_L(w_2)+f'_L(w_3) \leq 0$.

(1): We show that if $\mid R_a \mid R_b \mid R_c \mid$ is perimeter minimizing, then $V_a \geq V_b$
and $V_c \geq V_b$. Let the endpoints in volume coordinates be $w_0,w_1,w_2,w_3$.
Assume by symmetry that $w_1+w_2 \geq 0$. In particular, $w_2 \geq 0$. Now switching $R_b$ and $R_c$ gives either $V_c \geq V_b$ or $w_1+V_c \leq -w_2$.
But the second case implies that $w_1+w_2\leq -V_c<0$, contradiction. Hence $V_c \geq V_b$.

It remains to show that $V_a \geq V_b$.
Suppose for contradiction that $V_b>V_a$.
We know that $0 \in R_a \cup R_b$.
If $0 \in R_a$, then moving $R_c$ to the left of $R_a$ strictly decreases perimeter.
So $0 \in R_b$.
By switching $R_a$ and $R_b$, either $V_a \geq V_b$ or $w_2-V_a \geq -w_1$.
The first case is a contradiction, so the second case holds.
Because $w_1+w_2 \geq V_a > 0$, Lemma \ref{lem:summorethanzero} implies that $f'_L(w_1)+f'_L(w_2) \geq 0$.
Since $V_c \geq V_b > V_a$, $w_3-w_2>w_1-w_0$, and so $w_0+w_3>w_1+w_2 > 0$.
By Lemma \ref{lem:summorethanzero}, $f'_L(w_0)+f'_L(w_3) \geq 0$.
Now the equilibrium condition
$$f'_L(w_0)+f'_L(w_1)+f'_L(w_2)+f'_L(w_3) \leq 0$$
implies that $f'_L(w_1)+f'_L(w_2) = f'_L(w_0)+f'_L(w_3) = 0$.
Thus $f'_R(-w_1)=f'_L(w_2)$ and $f'_R(-w_0)=f'_L(w_3)$.

If the density is strictly log-convex, then $-w_1=w_2$ and $-w_0=w_3$, and so $V_a=V_c \geq V_b$, contradiction.
But further arguments are needed in the general case.
By convexity, $f$ in volume coordinates is linear on the intervals $[-w_1,w_2]$ and $[-w_0,w_3]$.
Recall that $w_0+w_3>w_1+w_2 > 0$.
Move all boundary points to the left while preserving volumes until $w_1+w_2=0$.
Along the way, if we use primes to indicate moved boundary points, $-w_1\leq -w_1' \leq w_2' \leq w_2$ and $-w_0 \leq -w_0' \leq w_3' \leq w_3$.
By linearity of $f$, perimeter does not change during the move.
But now switching $R_a$ and $R_b$ after the move strictly decreases perimeter, a contradiction. Therefore $V_a \geq V_b$.
\end{proof}

Observe that the conclusion of Proposition \ref{prop:typelogconvex} does not hold for a symmetric, log-convex density that is not uniquely minimized at the origin.
Indeed, such a density attains its minimum on some open interval containing the origin.
So the type $\mid R_1 \mid R_2 \mid R_3 \mid$ with all boundary points near the origin is perimeter minimizing.

\section{Numerical Plots of Minimizers}
\label{sec:code}

In the previous section, we show that there can only be four types of perimeter-minimizing triple bubbles
for a symmetric, log-convex density that is uniquely minimized at the origin.
We prove in this section that no more types can be ruled out, completing the proof of Theorem \ref{thm:main}.
Specifically, we show via numerical computation that for a certain symmetric, strictly log-convex, $C^1$ density, each of the four types of Proposition \ref{prop:typelogconvex} is perimeter minimizing for some prescribed volumes.
The result is conclusive since winning margins are larger than numerical errors; see Table \ref{tab:fourtypes}.
We designate the four types as 213, 3123, 2313, and 32123, as in Figure \ref{fig:4types}.

\subsection{Implementation}
We outline how to compute equilibria of the four types.
The following proposition shows how to compute a standard $n$-bubble for given volumes.

\begin{prop}
\label{prop:computestandard}
On $\R$ with a symmetric, strictly log-convex, $C^1$ density $f$, given volumes $V_1,\dots,V_n$, a standard $n$-bubble enclosing volumes
$V_1,\dots,V_n$ from left to right in this order exists and is unique.
In volume coordinates, its leftmost boundary point $\Vtil$ is the unique solution to
\begin{equation}
\label{eq:equigeneral}
f'(\Vtil)+f'(\Vtil+V_1)+\dots+f'(\Vtil+V_1+\dots+V_n)=0, \quad \Vtil \in [-V_1-\dots-V_n,0].
\end{equation}
\end{prop}

\begin{proof}
Equation \eqref{eq:equigeneral} is the equilibrium condition \cite[Cor. 3.3]{Bo}.
The conclusion follows from the fact that its left-hand side is strictly increasing in $\Vtil$ due to strict convexity, is negative for $\Vtil = -V_1-\dots-V_n$, and is positive for $\Vtil = 0$.
\end{proof}

From Proposition \ref{prop:computestandard}, if $f'$ is an elementary function in volume coordinates, we can effciently solve for a standard $n$-bubble using bisection search.

The next proposition considers existence and uniqueness of equilibria of the four types.

\begin{prop}
\label{prop:fourtypesexist}
Consider $\R$ with a symmetric, strictly log-convex, $C^1$ density $f$ and prescribed volumes $V_1\leq V_2 \leq V_3$.
An equilibrium of each of the types $213$, $3123$, and $32123$ exists and is unique.
An equilibrium of type $2313$ exists and is unique if, in volume coordinates,
$$f'\paren{V_3+\frac{V_1}{2}} > f'\paren{V_2+\frac{V_1}{2}} + f'\paren{\frac{V_1}{2}},$$
and does not exist otherwise.
\end{prop}

\begin{proof}
Notice that each equilibrium is a standard nested bubble. By Proposition \ref{prop:computestandard}, its constitutent standard bubbles exist and are unique.
So an equilibrium exists and is unique if the standard bubbles are properly nested, and does not exist otherwise.

Equilibria of types $213$ and $32123$ are easily seen to be properly nested.
To check that an equilibrium of type $3123$ is properly nested, apply \cite[Lemma 4.8]{Bo} to the inner standard bubble.
Finally, to derive conditions for an equilibrium of type $2313$ to be properly nested,
use the equilibrium condition \eqref{eq:equigeneral} for the outer standard bubble
and the fact that its left-hand side is strictly increasing in $\Vtil$.
\end{proof}

Given a density and prescribed volumes, we can compute perimeters of equilibria of the four types
by first checking for existence using Proposition \ref{prop:fourtypesexist} and then solving for the constituent standard bubbles using Proposition \ref{prop:computestandard}.
We can then generate 3D plots of types of perimeter minimizers as prescribed volumes vary, as shown in Figures \ref{fig:firstdensity} and \ref{fig:seconddensity}.
A 3D plot is represented as an animation of 2D plots, with $V_2$ and $V_3$ being spatial axes and $V_1$ varying in time.

We plot only the region where $V_1 \leq V_2 \leq V_3$, and the complement of this region is left as an empty space.
To prevent rounding errors, we plot a point only if the perimeter-minimizing type has perimeter at least $10^{-4}$ less than those of the other types, and we leave the point an empty space otherwise.
In the top left plot of Figure \ref{fig:seconddensity}, the white stripe between the \textcolor{teal}{green} and \textcolor{violet}{purple} regions is a result of this.

Solving for equilibria using volume coordinates is more efficient than doing so using positional coordinates as in our previous paper \cite{Bo}.
Indeed, a plot in this paper has more than $10^7$ points and the plot in our previous paper \cite[Fig. 9]{Bo} has $10^4$ points, but both take on the order of 30 minutes to generate.
We also use a faster language, Julia, while our previous paper used Mathematica.
This accounts for some of the speedup.

\subsection{Results}

We generate plots of two densities, to illustrate two phenomena that can occur. The first density is
\begin{equation}
\label{eq:density1}
f_1(V)=\abs{V}\sqrt{\log(\abs{V}+1)}+1.
\end{equation}
We can check that $f_1$ is strictly log-convex and $C^1$. The idea behind $f_1$ is that we are trying
to simulate the Borell density $f(x)=e^{x^2}$ in the asymptotics.
By the Fundamental Bounding Lemma 6.3 in \cite{Bo}, the Borell density grows asymptotically
as $\abs{V}\sqrt{\log{\abs{V}}}$ in volume coordinates.
Then we put in the ``plus 1'' to make the density well-defined and positive.
For $f_1$, only 3 types of minimizers, except 2313, can be visually identified.
(This is not a proof that the other type does not occur).
Figure \ref{fig:firstdensity} provides snapshots of the plot and their descriptions.

\begin{figure}[h]
\includegraphics[width=225pt]{{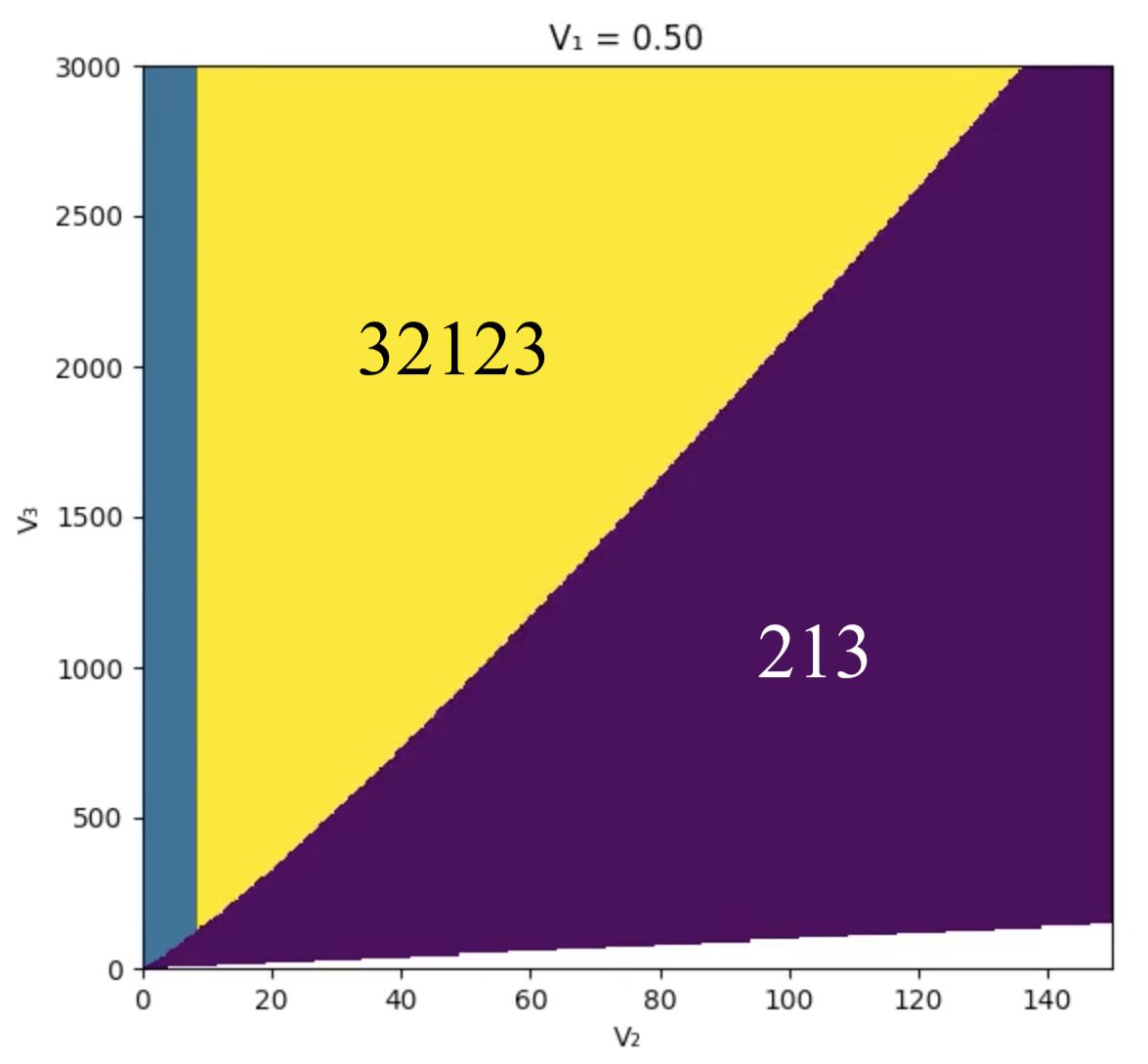}}
\includegraphics[width=225pt]{{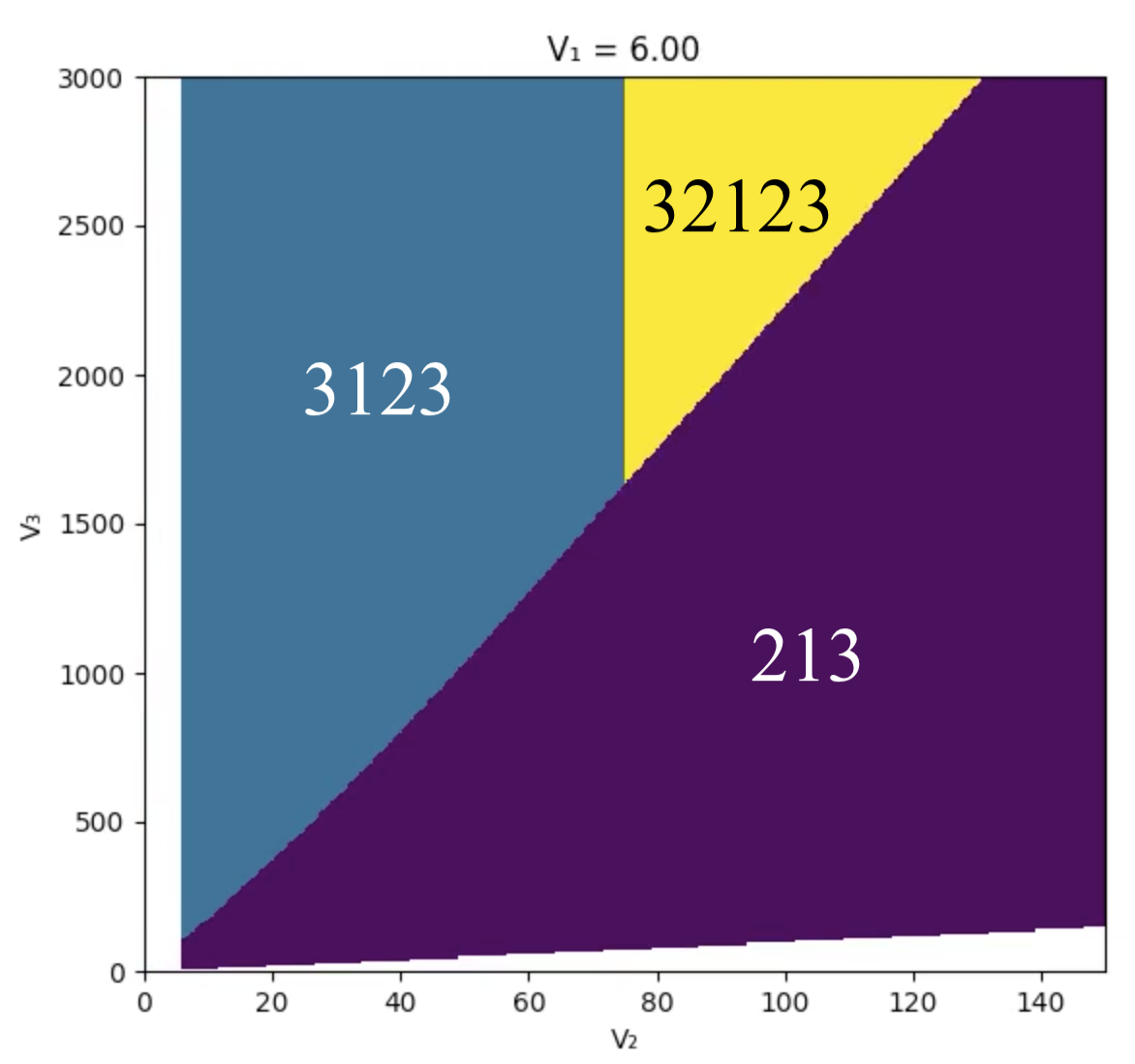}}
\caption{Snapshots from the animation (\url{https://github.com/natso26/triple-bubbles/raw/master/triple.mp4})
of 3D triple bubble type as a function of the three prescribed volumes $V_1$, $V_2$, $V_3$ for the density $f_1$ \eqref{eq:density1}.
Note that the $V_2$ and $V_3$ axes are on different scales.
Colors correspond to types as: \textcolor{violet}{purple} = 213; \textcolor{blue}{blue} = 3123; \textcolor{orange}{yellow} = 32123. The white areas near the axes are regions where $V_1 \leq V_2 \leq V_3$ is not satisfied and so nothing is drawn. As $V_1$ increases, the blue region pushes the yellow region to the right, while the purple region slowly rises.}
\label{fig:firstdensity}
\end{figure}

From the plot of the density $f_1$, one might wonder whether the type 2313 can ever be perimeter minimizing.
It turns out that this can happen if the density at the origin decreases. Consider the second density
\begin{equation}
\label{eq:density2}
f_2(V)=\abs{V}\sqrt{\log(\abs{V}+1)}+0.01,
\end{equation}
which is a translate of $f_1$.
For $f_2$, all four types in Proposition \ref{prop:typelogconvex} can be perimeter minimizing.
Nevertheless, the type 2313 can only be seen for small $V_1$, so the $V_1$ (time) axis is plotted on a logarithmic scale.
Figure \ref{fig:seconddensity} shows snapshots of the plot and their descriptions.
Based on the animation, I also attempt to sketch the 3D plot in Figure \ref{fig:sketch}.

Table \ref{tab:fourtypes} shows that all four types can be perimeter minimizing for the density $f_2$.

\begin{table}[h]
\begin{tabular}{|c|c|c|c|c|}
\hline 
$(V_1,V_2,V_3)$ & Type \textcolor{violet}{213} & Type \textcolor{blue}{3123} & Type \textcolor{teal}{2313} & Type \textcolor{orange}{32123} \\ 
\hline 
$(5, 100, 500)$ & \cellcolor{yellow} \textbf{1479.6294773} & 1667.8737745 & Not exist & 1661.4875997 \\ 
\hline 
$(5, 40, 2000)$ & 5608.7794571 & \cellcolor{yellow}  \textbf{5467.6249803} & Not exist & 5469.4347271 \\ 
\hline 
$(0.01, 100, 1500)$ & 4271.5195673 & 4351.3210336 & \cellcolor{yellow}  \textbf{4271.5168203} & 4335.5242035 \\ 
\hline 
$(2, 80, 2500)$ & 7167.5032872 & 7080.5694767 & Not exist & \cellcolor{yellow} \textbf{7071.1211666} \\ 
\hline 
\end{tabular}
\\
\caption{Prescribed volumes $(V_1,V_2,V_3)$ for which each of the four types are perimeter minimizing
for the density $f_2$ \eqref{eq:density2}. The numbers in the tables are perimeters of equilibria of each type.}
\label{tab:fourtypes}
\end{table}

In Figures \ref{fig:firstdensity} and \ref{fig:seconddensity}, the boundaries between the \textcolor{blue}{blue} and \textcolor{orange}{yellow} regions
are straight lines. In the $(V_1,V_2)$ coordinates, they are tie curves between the double bubbles $12$ and $212$
as studied in \cite{Bo}.

Code and animation for these two densities can be found at my GitHub repository \url{https://github.com/natso26/triple-bubbles}. The code is written in Julia. To run the code, follow the instructions in the repository to run it in \url{http://www.juliabox.com}.

\DeclareRobustCommand{\captionpar}{\par}

\begin{figure}[p]
\includegraphics[width=225pt]{{2-0.01.png}}
\includegraphics[width=225pt]{{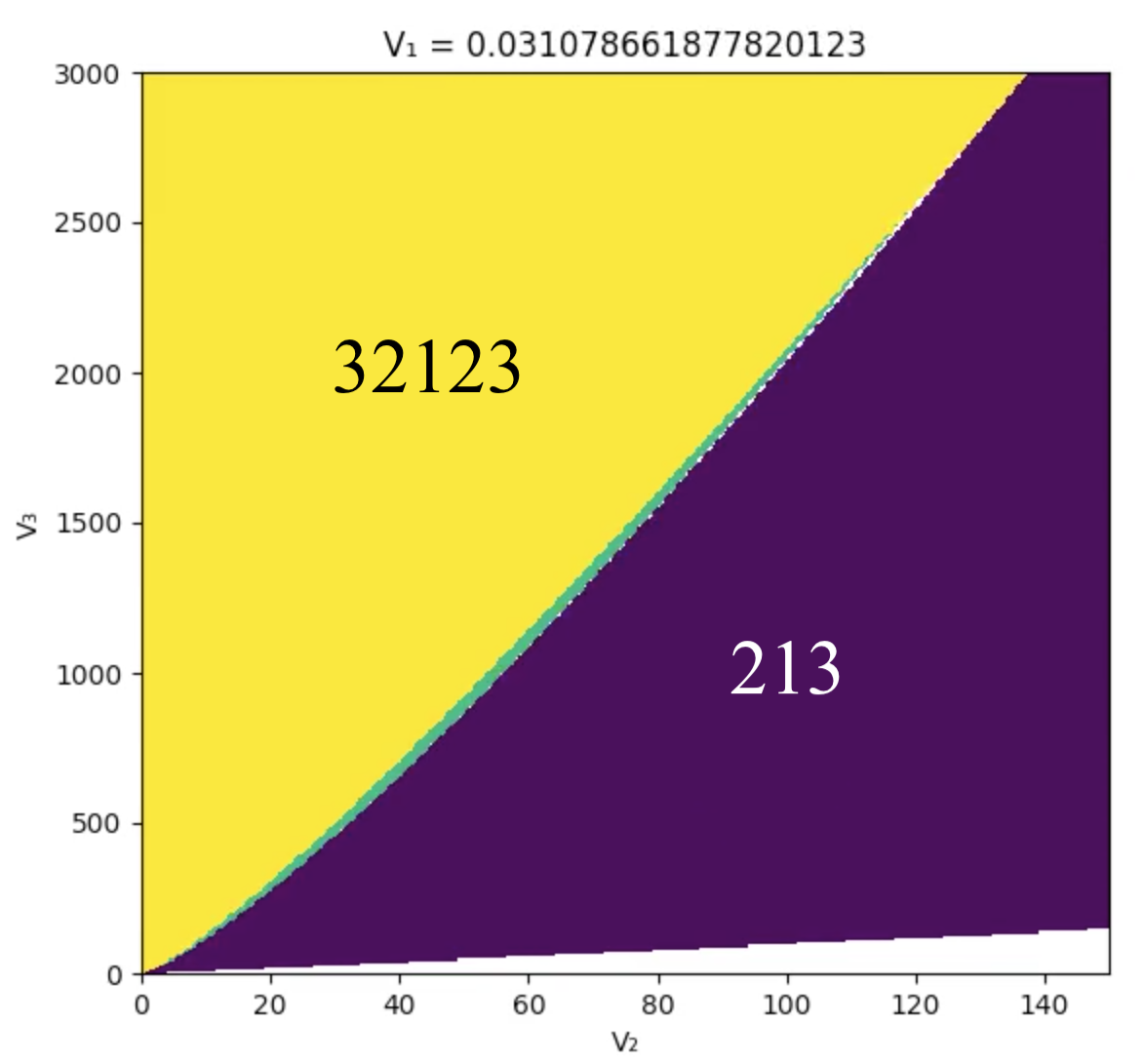}}
\includegraphics[width=225pt]{{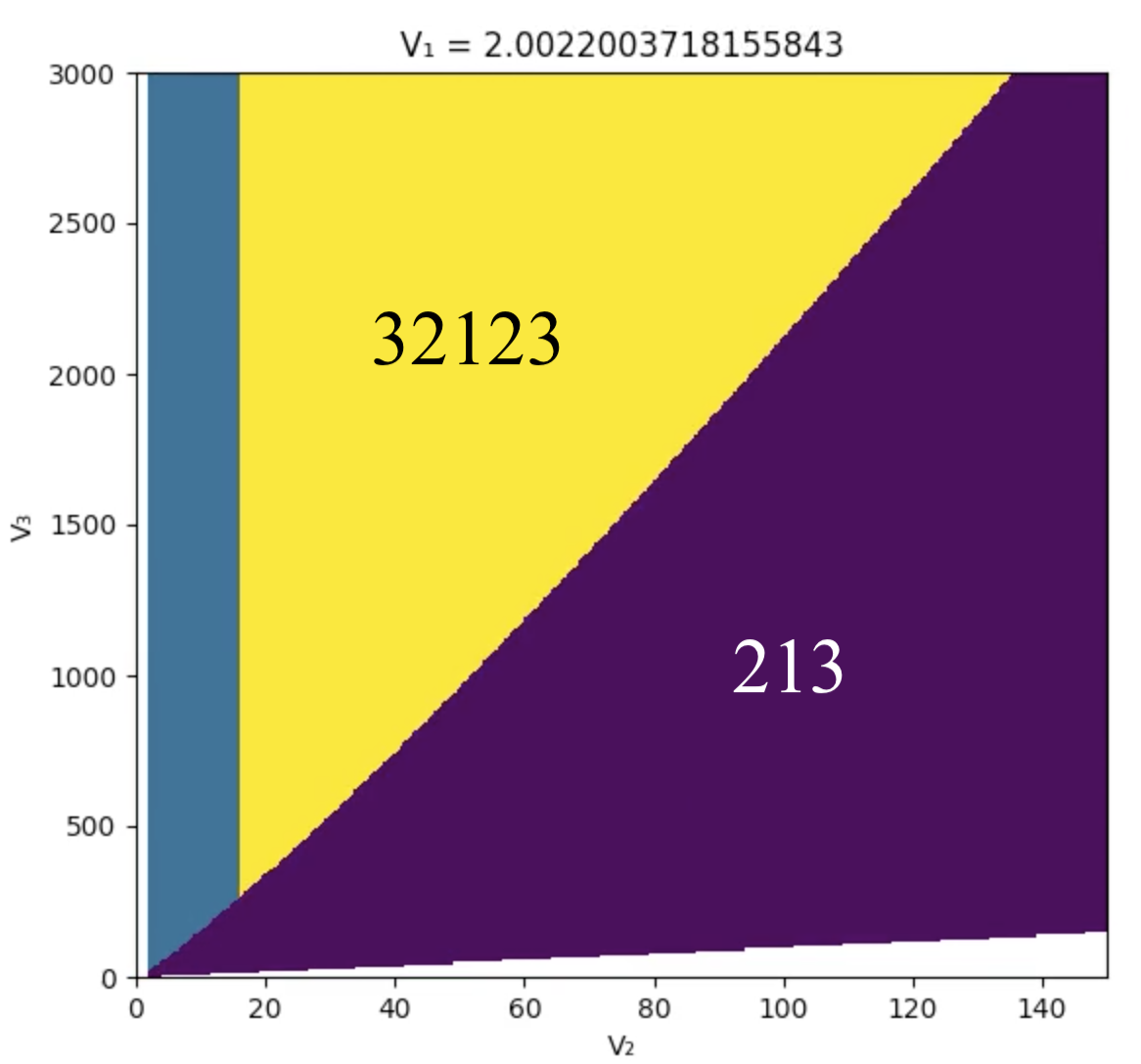}}
\includegraphics[width=225pt]{{2-6.png}}
\caption{Snapshots from the animation (\url{https://github.com/natso26/triple-bubbles/raw/master/triple2.mp4})
of 3D triple bubble type as a function of the three prescribed volumes $V_1$, $V_2$, $V_3$ for the density $f_2$ \eqref{eq:density2}.
The $V_2$ and $V_3$ axes are on different scales.
Colors correspond to types as: \textcolor{violet}{purple} = 213; \textcolor{blue}{blue} = 3123; \textcolor{teal}{green} = 2313; \textcolor{orange}{yellow} = 32123. The white areas near the axes are regions where $V_1 \leq V_2 \leq V_3$ is not satisfied and so nothing is drawn. The white stripes near the transition boundaries are regions where the difference between a perimeter-minimizing type and another type does not exceed $10^{-4}$ and so nothing is drawn to guard against rounding errors. \captionpar\setlength{\parindent}{1em}
In this animation, $V_1$ increases exponentially. As $V_1$ increases, the purple region first squeezes the green region out of existence; then the blue region emerges and pushes the yellow region to the right while the purple region slowly rises.
The green region can only be seen when $V_1$ is small (less than 0.05).}
\label{fig:seconddensity}
\end{figure}

\begin{figure}[p]
\includegraphics[width=400pt]{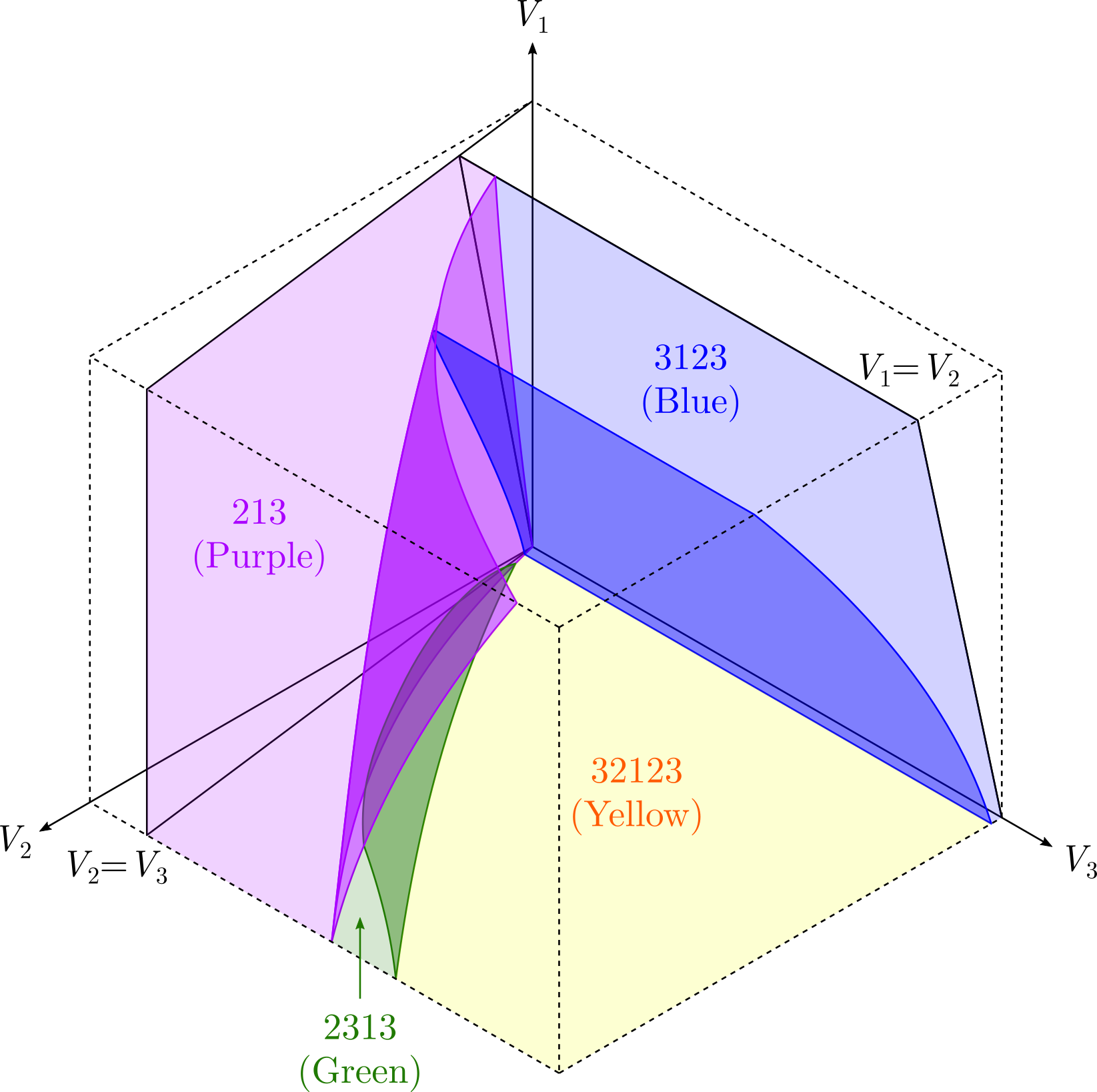}
\caption{A rough sketch of 3D triple bubble type as a function of the three prescribed volumes $V_1$, $V_2$, $V_3$ for the density $f_2$ \eqref{eq:density2}.
The axes are not on the same scale, and the scale for the $V_1$ axis is nonlinear
just as in the animation. I do not claim any accuracy of this sketch.
The transition boundaries seem to be made up of three surfaces stitched together.
The green (2313) region is possibly bounded; if so, this is the first instance of a bounded region
for 1D $n$-bubbles.
Section \ref{sec:conj} provides more conjectures.}
\label{fig:sketch}
\end{figure}

Therefore, based on Proposition \ref{prop:typelogconvex} and the work in this section, we conclude that Theorem \ref{thm:main} holds.

\section{Further Work}
\label{sec:conj}

In this work, we completely determined the possible types of perimeter-minimizing triple bubbles for symmetric, strictly log-convex densities.
Nevertheless, there remains many unanswered questions about transitions between different types.
In particular, many conjectures can be formulated based on the sketch in Figure \ref{fig:sketch}.
To list some questions and conjectures:

\begin{enumerate}
\item \textbf{Possible types.} For some densities, there are only three types of minimizers (all except 2313), as conjectured to be the case for the density $f_1$ \eqref{eq:density1}; for other densities, all four types of minimizers occur, as is the case for the density $f_2$ \eqref{eq:density2}.
Can there be other sets of possible perimeter-minimizing types?
\item \textbf{Nature of tie surfaces.} The tie surfaces, that is, the set where at least two types of minimizers have equal perimeter, are really surfaces: multiple 2D manifolds stitched together. Moreover, for $C^k$ densities, they are $C^k$ manifolds (except where they meet). Furthermore, are there only three surfaces stitched together as in Figure \ref{fig:sketch}, or is this an illusion created by some surfaces meeting at nearly 180 degrees?
\item \textbf{Boundedness of regions.} The region for type 2313, if it exists, is bounded.
The regions for other types are never bounded.
\item \textbf{When some volumes are equal.}
If $V_2=V_3$, then the minimizer is of type 213. If $V_1=V_2$, then as $V_3$ increases, the minimizer transitions from type 213 to type 3123.
\item \textbf{When volumes are large.}
For each $V_1$, for $V_2$ large (as a function of $V_1$), for $V_3$ large (as a function of $V_1$ and $V_2$),
the minimizer is of type 32123.
\item \textbf{Transitions when volumes increase.}
There is a constant $\lambda_0$, a function $\lambda_1(V_1)$, and functions $\lambda_{2313}(V_1)$ and $\lambda_{2313}'(V_1)$ for $V_1 \leq \lambda_0$
with the following properties:
\begin{itemize}
\item If $V_2 < \lambda_1(V_1)$, then as $V_3$ increases, the minimizer transitions from type 213 to type 3123.
\item If $V_2 > \lambda_1(V_1)$, $V_1<\lambda_0$, and $\lambda_{2313}(V_1) < V_2 <  \lambda_{2313}'(V_1)$,
then as  $V_3$ increases, the minimizer transitions from type 213 to type 2313 and then to type 3123.
\item If $V_2 > \lambda_1(V_1)$ and either $V_1>\lambda_0$, $V_2 < \lambda_{2313}(V_1)$, or $V_2 > \lambda_{2313}'(V_1)$, then the minimizer
transitions only from type 213 to type 3123.
\end{itemize}
See Figure \ref{fig:sketch}.
\end{enumerate}

\section*{Appendix}

The following are results from \cite{Bo} that are used in this work.

{\small
\begin{namedtheorem}[Corollary 3.3]
Let $f$ be a $C^1$ density on $\R$.
If an $n$-bubble with boundary points $x_1<x_2<\dots<x_k$ is perimeter minimizing, then
$$\sum_{i=1}^k (\log f)'(x_i)=0.$$
More generally, if $1 \leq a < b \leq k$ are such that
the blocks to the left of $x_a$ and to the right of $x_b$ both belong to the same bubble or to no bubble, then
$$\sum_{i=a}^b (\log f)'(x_i)=0.$$
\end{namedtheorem}

\begin{namedtheorem}[Remark 3.4]
In Corollary 3.3, if the condition on $f$ is relaxed from $C^1$ to one-sided derivatives (for example if $f$ is convex or log-convex), then similarly the sum of the right derivatives is nonnegative and the sum of the left derivatives nonpositive.  
\end{namedtheorem}

\begin{namedtheorem}[Proposition 3.5]
On $\R$ with a continuous density
that is nonincreasing on $(-\infty,0]$ and 
nondecreasing on $[0,\infty)$,
a perimeter-minimizing $n$-bubble
consists of finitely many \emph{contiguous} intervals. 
\end{namedtheorem}

\begin{namedtheorem}[Lemma 3.7]
Consider $\R$ with a continuous density that is nonincreasing on $(-\infty,0]$ and nondecreasing on $[0,\infty)$. Let $M$ be the density minimum set where $f(x)=f(0)$. Consider two components of the same bubble in a perimeter-minimizing $n$-bubble.  Then the component on the right contains no points to the left of $M$ and some to the right of $M$. Similarly the component on the left contains no points right of $M$ and some left of $M$.
\end{namedtheorem}

\begin{namedtheorem}[Proposition 3.8]
On $\R$ with a continuous density that is nonincreasing on $(-\infty,0]$ and nondecreasing on $[0,\infty)$, a perimeter-minimizing $n$-bubble has at most $2n-1$ components.
\end{namedtheorem}

\begin{namedtheorem}[Definition 4.1]
A \emph{double interval $(x_1,x_2,x_3)$} for prescribed volumes
 $V_1\leq V_2$ consists of two contiguous intervals $[x_1, x_2]$, $[x_2, x_3]$ of volumes $V_1$ and $V_2$, respectively, as in Figure \ref{fig:doubletriple} (figure in this article). For a $C^1$ density $f$,
a double bubble is \emph{in equilibrium}
if it satisfies the consequence of perimeter minimization of Corollary 3.3:
$$(\log f)'(x_1)+(\log f)'(x_2)+(\log f)'(x_3)=0.$$ The term also applies to the generalization to one-sided derivatives of Remark 3.4.

The \emph{triple interval $(y_1,y_2)$} for prescribed volumes
$V_1 \leq V_2$ consists of three contiguous intervals, two of which flank the middle interval and enclose an equal volume, as in Figure \ref{fig:doubletriple} (figure in this article). The middle interval is $[-y_1, y_1]$ and encloses volume $V_1$. The left interval is $[-y_2, -y_1]$ and the right interval is $[y_1, y_2]$, and each encloses volume $V_2/2$.

For a symmetric continuous, piecewise $C^1$ density, the triple interval is in equilibrium.
\end{namedtheorem}

\begin{namedtheorem}[Lemma 4.2 \normalfont{(Volume coordinate)}]
On $\R$ with density f, let
$$V = \int_0^x f.$$
Then $f$ is a log-convex function of $x$ if and only if $f$ is a convex function of $V$.
\end{namedtheorem}

\begin{namedtheorem}[Proposition 4.6]
On $\R$ with symmetric, strictly log-convex density $f$,
for prescribed volumes $V_1\le V_2$,
a perimeter-minimizing double bubble is one of the following:
\begin{enumerate}[label = (\alph*)]
\item the unique double interval $(x_1,x_2,x_3)$ in equilibrium (up to reflection) or
\item the triple interval $(y_1,y_2)$.
\end{enumerate}
\end{namedtheorem}

\begin{namedtheorem}[Figure 9] \ \\
\begin{center}
\includegraphics[width=0.5\textwidth]{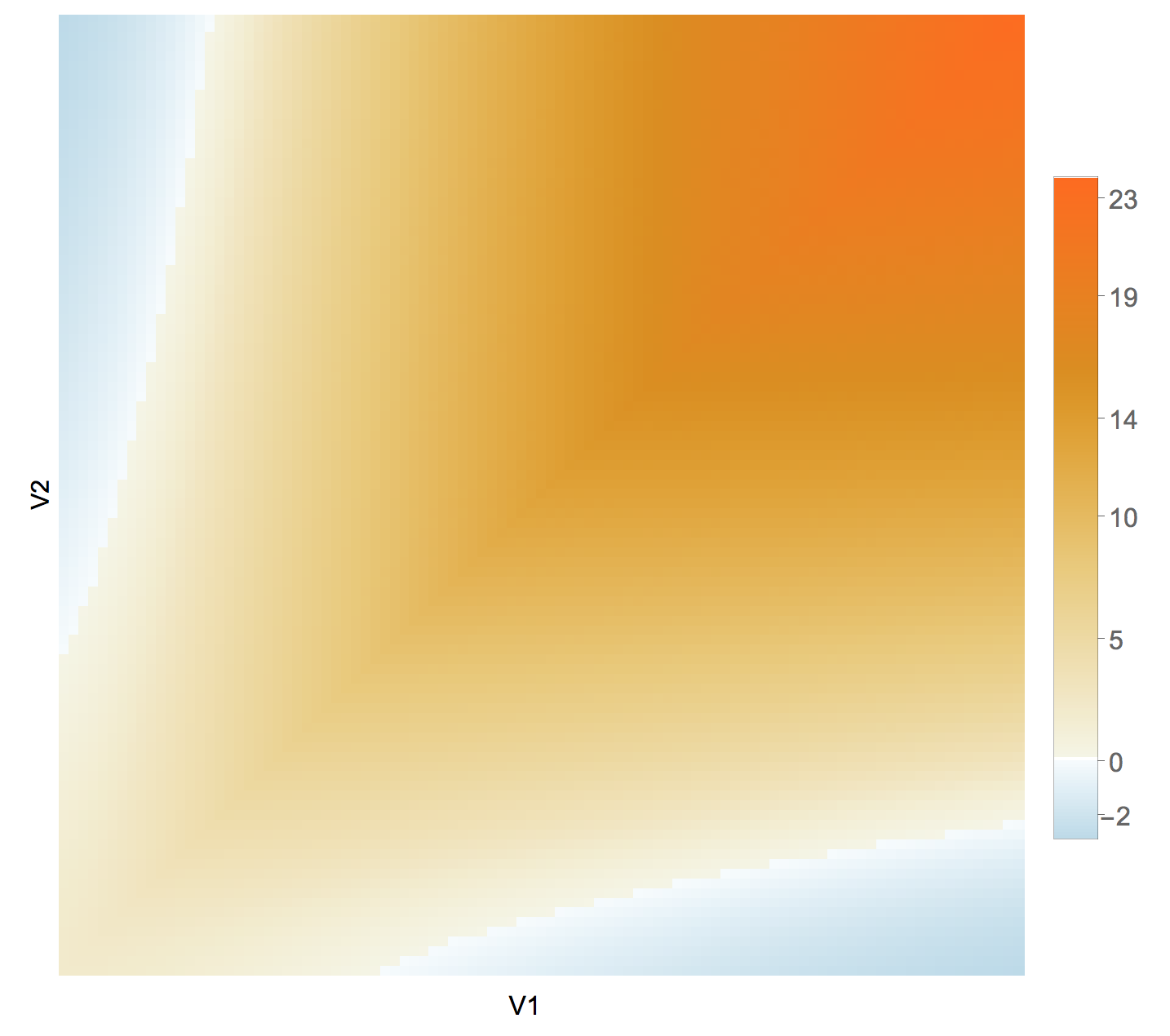}
\end{center}
\noindent
The figure above is a numerical computation which represents the value of the perimeter difference $\mu(V_1, V_2)$ for Borell density $f(x)=e^{x^2}$. The orange color marks the region in which the double interval has lesser perimeter, the blue color represents the region in which the triple interval has lesser perimeter, and the white curve marks the tie point between the double and triple intervals. Computed in Mathematica.
\end{namedtheorem}

\begin{namedtheorem}[Lemma 4.8]
On $\R$ with symmetric, strictly log-convex, $C^1$ density,
for prescribed volumes $V_1< V_2$,
$$-\frac{V_1+V_2}{2}<\widetilde{V}<-V_1.$$
\end{namedtheorem}

\begin{namedtheorem}[Theorem 4.15]
On $\R$ with symmetric, strictly log-convex, $C^1$ density $f$ such that $(\log f)'$ is unbounded,
given $V_1>0$, there is a unique $V_2=\lambda(V_1)$
such that the double interval in equilibrium and the triple interval
tie.
For $V_2>\lambda(V_1)$, the perimeter-minimizing
double bubble is uniquely the triple interval.
For $V_2<\lambda(V_1)$, the perimeter-minimizing
double bubble is uniquely the double interval
in equilibrium.
Moreover, $\lambda$ is a strictly increasing $C^1$ function
that tends to a positive limit as $V_1\to 0$.
\end{namedtheorem}

\begin{namedtheorem}[Conjecture 7.1]
In $\mathbb{R}^N$ with a smooth, radial, log-convex density, a perimeter-minimizing double bubble is either
\begin{enumerate}[label=(\roman*)]
\item the bubble inside a bubble (e.g. for $V_1$ small and $V_2$ large), or
\item the standard double bubble (e.g. for $V_2$ close to $V_1$).
\end{enumerate}
\end{namedtheorem}
}


\begin{thebibliography}{10}

\bibitem[Ba]{Ba} Vincent Bayle, \emph{Propri\'{e}t\'{e}s de concavit\'{e} du profil isop\'{e}rim\'{e}trique et applications}, PhD thesis, Institut Joseph Fourier, Grenoble (2004).

\bibitem[BH]{BH} Serguei G. Bobkov, Christian Houdr\'{e}, \emph{Some connections between isoperimetric and Sobolev-type inequalities}, Mem. Amer. Math. Soc. 616 (1997), 1--111.

\bibitem[Bo]{Bo} Eliot Bongiovanni, Alejandro Diaz, Leonardo Di Giosia, Jay Habib, Arjun Kakkar, Lea Kenigsberg, Dustin Ping, Dylanger Pittman, Nat Sothanaphan, Weitao Zhu, \emph{Double bubbles on the line with log-convex density}, Anal. Geom. Metric Spaces 6 (2018), 64--88, \url{https://arxiv.org/abs/1708.03289}.

\bibitem[Ch]{Ch} Gregory R. Chambers, \emph{Proof of the Log-Convex Density Conjecture}, J. Eur. Math. Soc. (2015), to appear, \url{https://arxiv.org/abs/1311.4012}.

\bibitem[HMRR]{HMRR} Michael Hutchings, Frank Morgan, Manuel Ritor\'e, Antonio Ros, \emph{Proof of the
double bubble conjecture}, Ann. of Math., 155 (2002), no. 2, 459--489, \url{https://arxiv.org/abs/math/0406017}.

\bibitem[MN]{MN} Emanuel Milman, Joe Neeman, \emph{The Gaussian double-bubble conjecture}, preprint (2018), \url{https://arxiv.org/abs/1801.09296}.

\bibitem[MN2]{MN2} Emanuel Milman, Joe Neeman, \emph{The Gaussian multi-bubble conjecture}, preprint (2018), \url{https://arxiv.org/abs/1805.10961}.

\bibitem[Mo]{Mo} Frank Morgan, \emph{Geometric Measure Theory: A Beginner's Guide}, Academic press (2016).

\bibitem[RCBM]{RCBM} C\'esar Rosales, Antonio Ca\~nete, Vincent Bayle, Frank Morgan, \emph{On the isoperimetric problem in Euclidean space with density}, Calc. Var. PDE 31 (2008), 27--46, \url{https://arxiv.org/abs/math/0602135}.

\bibitem[Re]{Re} Ben W. Reichardt, \emph{Proof of the double bubble conjecture in $\R^n$}, J. Geom. Anal. 18 (2008), no. 1, 172--191, \url{https://arxiv.org/abs/0705.1601}.

\bibitem[So]{So} Nat Sothanaphan, \emph{Double Bubbles on the Line with Log-Convex Density $f$ with $(\log f)'$ Bounded}, Missouri J. Math. Sci. 30 (2018), no. 2, 166--175, \url{https://arxiv.org/abs/1807.02661}.

\bibitem[Wi]{Wi} Wacharin Wichiramala, \emph{Proof of the planar triple bubble conjecture}, J. Reine Angew. Math. 567 (2004), 1--49.

\end{thebibliography}
\end{document}